\documentclass[12pt]{amsart}

\usepackage[margin=1in]{geometry}
\usepackage{amsmath, amsthm, amsfonts}
\usepackage{wrapfig}
\usepackage{amsfonts}
\usepackage{amsmath}
\usepackage{amssymb}

\usepackage{color}
\usepackage{comment}
\usepackage{tikz}

\usepackage{MnSymbol}

\newtheorem{Theorem}{Theorem}[section]
\newtheorem{Conjecture}[Theorem]{Conjecture}
\newtheorem{Remark}[Theorem]{Remark}
\newtheorem{Lemma}[Theorem]{Lemma}
\newtheorem{Proposition}[Theorem]{Proposition}
\newtheorem{Definition}[Theorem]{Definition}
\newtheorem{Corollary}[Theorem]{Corollary}

\newcommand{\lt}{\left}
\newcommand{\rt}{\right}
\newcommand{\bpm}{\begin{pmatrix}}
\newcommand{\epm}{\end{pmatrix}}
\newcommand{\bsm}{\lt(\begin{smallmatrix}}
\newcommand{\esm}{\end{smallmatrix}\rt)}

\newcommand{\ZZ}{\ensuremath{\mathbb{Z}}}
\newcommand{\N}{\ensuremath{\mathbb{N}}}

\newcommand{\R}{\ensuremath{\mathbb{R}}}
\newcommand{\C}{\ensuremath{\mathbb{C}}}

\newcommand{\CC}{\mathbb{C}}

\newcommand{\cK}{\mathcal{K}}
\newcommand{\cO}{\mathcal{O}}

\newcommand{\cL}{\mathcal{L}}
\newcommand{\cJ}{\mathcal{J}}

\newcommand{\fF}{\mathfrak{F}}

\newcommand{\h}{\ensuremath{\mathfrak{h}}}
\newcommand{\g}{\ensuremath{\mathfrak{g}}}

\renewcommand{\O}{\mathcal{O}}

\newcommand{\Gr}{\mathcal{G}r}
\newcommand{\Gm}{G_1[[t^{-1}]]}
\newcommand{\res}{\text{Res}}
\newcommand{\Y}{\mathcal{Y}}
\renewcommand{\det}{\text{det}}
\DeclareMathOperator{\Spec}{Spec}
\newcommand{\field}{\CC}
\newcommand{\meet}{\wedge}
\DeclareMathOperator{\val}{val}

\title{On a reducedness conjecture for spherical Schubert varieties and slices in the affine Grassmannian}

\author{Joel Kamnitzer}
\address {Department of Mathematics, 
University of Toronto,
Toronto, ON, Canada M5S 2E4}
\email{jkamnitz@math.toronto.edu}

\author{Dinakar Muthiah}
\address{Department of Mathematical and Statistical Sciences,
University of Alberta, Edmonton, AB, Canada T6G 2G1}
\email{muthiah@ualberta.ca}

\author{Alex Weekes}
\address {Department of Mathematics, 
University of Toronto,
Toronto, ON, Canada M5S 2E4}
\email{alex.weekes@mail.utoronto.ca}

\begin{document}
\maketitle

\begin{abstract}
We study spherical Schubert varieties in the affine Grassmannian.  These Schubert varieties have a natural conjectural modular description due to Finkelberg-Mirkovi\'c.  This modular description is easily seen to be set-theoretically correct, but it is not obviously scheme-theoretically correct.  We prove that this modular description is correct in many cases.  We also link this modular description to the reducedness conjecture from Kamnitzer-Webster-Weekes-Yacobi for tranverse slices in the affine Grassmannian.
\end{abstract}

\section{Introduction}

\subsection{The affine Grassmannian and its spherical Schubert varieties}

The affine Grassmannian $Gr_G$ associated to a semisimple group $G$ has two incarnations. It can be viewed as the quotient $G( \C((t))) / G(\C [[t]])$ or it can be viewed as a partial flag variety associated to an affine Kac-Moody group. Both points of view have their advantages, and much of the richness of the affine Grassmannian arises by playing these two points of view off one another. The presentation as a quotient $G( \C((t))) / G(\C [[t]])$ is closely related to the following modular description of the affine Grassmannian. Fix a smooth algebraic curve $X$ and a point $x \in X$. The affine Grassmannian is the moduli space of pairs $(P,\phi)$ where $P$ is a principal $G$ bundle on $X$, and $\phi$ is a trivialization of $P$ away from $x$, where such pairs are considered up to isomorphism.

It is therefore natural to ask how much of the geometry of the affine Grassmannian can be interpreted in modular terms. Specifically, we want to study $G(\C [[t]])$-orbit closures.  The orbits $ \Gr^\lambda $ are naturally indexed by dominant coweight for the group $G$. These orbit closures $\overline{\Gr^\lambda}$ are of signficant interest in geometric representation theory. For example, by the geometric Satake correspondence of Lusztig, Ginzburg, Beilinson-Drinfeld and Mirkovic-Vilonen, the intersection cohomology of these orbit closures carry an action of Langlands dual group. From the Kac-Moody point of view, the schemes $\overline{\Gr^\lambda}$ are Schubert varieties and thus are well understood by general theory of flag varieties associated to Kac-Moody groups and their Schubert varieties (e.g. \cite{Kum}).  

Finkelberg and Mirkovi\'c \cite{FM} propose a modular description $\overline{\Y^\lambda}$ of these orbit closures.  They define $ \overline{\Y^\lambda} $ to be the locus of pairs $ (P, \phi) $ where the poles of $ \phi $ at $ x $ are controlled by $ \lambda$.  It is easy to verify that this moduli space is set-theoretically supported on the Schubert varieties, but it is not at all clear that this moduli space is a reduced scheme. Studying the reducedness of this scheme is our first goal in this paper. The natural conjecture is the following:

\begin{Conjecture}{\label{conj:reducedness-of-orbit-closures}}
Let $\lambda$ be a dominant coweight for $G$. Then $\overline{\Y^\lambda}$ is reduced.
\end{Conjecture}

\subsection{Reducedness results}

In this paper, we prove a number of results related to this conjecture.

Our first result is the following.
\begin{Theorem} \label{th:genreducedintro}
For any $ \lambda $, $\overline{\Y^\lambda} $ is smooth (and therefore reduced) along $ \Gr^\lambda \subset \overline{\Gr^\lambda} $ .
\end{Theorem} 
\noindent Because it is known that $\overline{\Gr^\lambda}$ is Cohen-Macaulay, we also obtain the following.
\begin{Theorem}
  The scheme $\overline{\Y^\lambda}$ is reduced if and only if it is Cohen-Macaulay. 
\end{Theorem}

Then we specialize to the case $ G = SL_n $, and we prove the following result.
\begin{Theorem}
Let $G = SL_n$ and let $\mathcal{X}$ denote the set of all dominant coweights that can be written in the form $a \varpi_i + b \varpi_{i+1}$ for $a, b \geq 0$ and $i \in \{ 1, \ldots, {n-2} \}$. Then for all $\lambda \in \mathcal{X}$, $\overline{\Y^\lambda}$ is reduced.
\end{Theorem}
\noindent In particular, this covers all cases for $SL_2$ and $SL_3$.

\subsection{Affine Grassmannian slices}
An expanded form of the reducedness conjecture was stated by Kamnitzer-Webster-Weekes-Yacobi in \cite{KWWY}. There the authors study a particular Poisson structure on $\Gr$. This Poisson structure is related to the Lie bialgebra structure on $\mathfrak{g}[[t]]$, the quantization of which gives rise to the Hopf algebra structure on the Yangian. In particular, they study schemes $\Gr^\lambda_\mu$ which are canonically-defined (after choice of coordinate $t$) transverse slices to $\Gr^\mu$ inside of the orbit closure $\overline{\Gr^\lambda}$. The subschemes $\Gr^\lambda_\mu$ preserve the Poisson structure on $\Gr$, and the authors of \cite{KWWY} construct a quantization of a subscheme $\Y^\lambda_\mu$ that is set-theoretically supported on $\Gr^\lambda_\mu$, i.e. up to nilpotents. This scheme $\Y^\lambda_\mu$ has an explicit ring-theoretic description in terms of the representation theory of $G$ and Poisson brackets. The authors conjecture the following.

\begin{Conjecture}{\label{conj:reducedness-of-transverse-slices}}
Let $\lambda$ and $\mu$ be dominant coweights for $G$ such that $\mu \leq \lambda$. Then $\Y^\lambda_\mu = \Gr^\lambda_\mu$ as schemes. Equivalently, $\Y^\lambda_\mu$ is reduced.
\end{Conjecture}

It is easy to show that the above Conjecture \ref{conj:reducedness-of-transverse-slices} for $\mu=0$ is equivalent to the above Conjecture \ref{conj:reducedness-of-orbit-closures}. We also show a strong converse result:

\begin{Theorem}{\label{thm:reducedness-of-transverse-slices}}
Let $\lambda$ and $\mu$ be dominant coweights for $G$ such that $\mu \leq \lambda$. If $\overline{\Y^\lambda}$ is reduced, then Conjecture \ref{conj:reducedness-of-transverse-slices} is true for $\Y^\lambda_\mu$.
\end{Theorem}
In particular, for $SL_n$, Conjecture \ref{conj:reducedness-of-transverse-slices} is true for all $\lambda \in \mathcal{X}$.

\subsection{Overview}

The paper divides naturally into two parts. In the first part we study Conjecture \ref{conj:reducedness-of-orbit-closures}.  First, we prove Theorem \ref{th:genreducedintro} by computing the tangent space of $ \overline{\Y^\lambda} $ at the point $t^\lambda $ and showing that it has same dimension as $ \Gr^\lambda$.

 Next, we study a factorable version of $\overline{\Y^\lambda}$, i.e. a version of $\overline{\Y^\lambda}$ that naturally sits as a closed subscheme of the Beilinson-Drinfeld Grassmannian. Let $\lambda$ and $\mu$ be dominant weights, we prove that the reducedness of $\overline{\Y^{\lambda + \mu}}$ implies the reducedness of $\overline{\Y^\lambda}$ and $\overline{\Y^\mu}$. Using the fact that the scheme-theoretic
intersection of Schubert varieties is reduced, we conclude that if $\overline{\Y^\lambda}$ and $\overline{\Y^\mu}$ are both reduced, then so is $\overline{\Y^{\lambda \meet \mu}}$ (see Section \ref{sec:Some simplifying arguments} for the definition of $\lambda \meet \mu$). Thus we have two methods for proving reducedness using the reducedness in known cases.

For $SL_n$, we prove explicitly that $\overline{\Y^{nk\varpi_1}}$ is reduced by first showing that it is Cohen-Macaulay (Section \ref{sec:Type A calculations}). Then we show that it is generically reduced by exhibiting a smooth point. Using the $SL_n$ diagram automorphism, we obtain the same for $\overline{\Y^{nk\varpi_{n-1}}}$. Finally, using the two methods above, we prove that $\overline{\Y^{\lambda}}$ is reduced for all $\lambda \in \mathcal{X}$ (Section \ref{section: Generating more cases}).

In the second part of the paper (Sections \ref{section: Ideal generators for Lusztig slices} and \ref{sec: Poisson brackets}), we reduce Conjecture \ref{conj:reducedness-of-transverse-slices} to Conjecture \ref{conj:reducedness-of-orbit-closures}. This involves explicitly studying the Poisson ideal that defines $\Y^\lambda_\mu$. In particular, we give generators for the ideal as an ordinary ideal (without reference to the Poisson structure) and relate them to equations defining $\overline{\Y^\lambda}$. 
\subsection{Acknowledgements}
We thank Alexander Braverman, Shrawan Kumar, Oded Yacobi, and Xinwen Zhu for helpful conversations. We especially thank Xinwen Zhu for pointing out an error in a previous version of this paper.  J.K. was supported by NSERC and a Sloan Fellowship. D.M. was supported by a PIMS Postdoctoral Fellowship.

\section{Affine Grassmannian and $G(\cO)$-orbit closures}
\label{Affine grassmannian and G(O) orbit closures}

\subsection{Notation}
\label{Notation}
Throughout the paper, $G$ denotes a fixed semisimple group and $\g$ its Lie algebra. For simplicity we will assume $G$ is simply-connected. We also work throughout over the field $\CC$ of complex numbers. We remark that the results in Sections \ref{sec:tangent-space-computation} and \ref{sec:beilinson-drinfeld-deformation} still hold if we replace $\CC$ with any algebraically closed field of arbitrary characteristic. 

We write $ \lambda, \mu$, etc. for coweights of $ G $ and $ \nu^\vee$, etc. for weights of $ G $.  Denote their pairing by $\langle \lambda, \nu^\vee\rangle$.  We write $ \varpi_i, \varpi_i^\vee $ for the fundamental coweights and weights of $  \g$.  Similarly, we write $ \alpha_i $ for the simple coroots of $ \g $ and $ \alpha_i^{\vee} $ for the simple roots.  We write $ \Phi^\vee $ for the set of all roots.  

Let $ \lambda $ be a dominant coweight of $ G $.  Then we write $ \lambda^* = -w_0 \lambda $, where $ w_0 $ denotes the longest element of the Weyl group of $ G $.

Let $ a_{ij} $ denote the Cartan matrix of $ \g $, and let $ d_i $ be the positive integers chosen to symmetrize the Cartan matrix.  Choose Chevalley generators $e_i, h_i, f_i$ for $\g$, and extend this to a choice of root vectors $e_{\alpha^\vee}, f_{\alpha^\vee}$ for $\alpha^\vee \in \Phi_+^\vee$. We choose a $\g$-invariant symmetric bilinear form  $(\cdot,\cdot)$ on $ \g $, so that $(h_i, h_j) = d_i a_{ij}$ and $(e_i, f_j) = d_i^{-1} \delta_{ij}$, as in \cite[Chapter 2]{Kac}.  Denote $(e_{\alpha^\vee}, f_{\alpha^\vee}) = d_{\alpha^\vee}^{-1}$, and choose a dual basis $\{h^i\}$ to $\{h_i\}$.  Identifying $\h$ and  $\h^\ast$ via $(\cdot,\cdot)$ there is an induced bilinear form $(\cdot,\cdot)$ on $\h^\ast$, which for $\nu^\vee,\eta^\vee\in\h^\ast$ is given by
$$ (\nu^\vee,\eta^\vee) = \sum_i \nu^\vee(h_i)\eta^\vee(h^i)$$

\subsection{Modular description of $\Gr$}

 Recall the following modular description of the affine Grassmannian $\Gr$. Let $X$ be a a smooth curve, and let $x \in X$ be a closed point. An $S$-point of $\Gr_{X,x}$ consists of

\begin{itemize}
\item$\fF_G$ a $G$-bundle on $S \times X$.
\item A trivialization
$\phi : \fF^0_G \dashedrightarrow \fF_G$ defined on $S\times (X-x)$
\end{itemize}

By the Beauville-Laszlo theorem, this definition only depends on the adic disk centered at $x$, and we can identify $\Gr_{X,x}$ with the quotient $G(\cK _x)/G(\cO_x)$, where $\cO_x$ is the completed local ring at $x$, and $\cK_x$ is the fraction field of $\cO_x$.  If we choose a local coordinate near $ x$, then we have isomorphisms $ \cK_x = \C((t)) $ and $ \cO_x = \C[[t]] $.  This gives us an isomorphism $ \Gr_{X,x} \cong \Gr := G(\C((t))) / G(\C[[t]]) $.

\subsection{A modular description of orbit closures}

The $G(\C[[t]])$-orbits on $\Gr$ are naturally indexed by the dominant coweights.  Let $ \lambda $ be a dominant coweight and let $ t^\lambda $ be the corresponding point in $ \Gr $.  Let us write $\Gr^\lambda := G(\C[[t]]) t^\lambda $ for the orbit through $\lambda$, and let $\overline{\Gr^\lambda}$ be the closure with the \emph{reduced} scheme structure.

We define $\overline{\Y^\lambda}$ to be the following closed subfunctor of the affine Grassmannian.

An $S$-point of $\overline{\Y^\lambda}$ consists of the following
\begin{itemize}
\item$\fF_G$ a $G$-bundle on $S \times X$.
\item A trivialization $\phi : \fF^0_G \dashedrightarrow \fF_G$ defined on $S\times (X-x)$

\item For every dominant weight $\nu^\vee$, the composed map
  \begin{align}
    \label{eq:4}
\phi_{\nu^\vee} : \fF^0_G \times^G V(\nu^\vee) \dashedrightarrow \fF_G \times^G V(\nu^\vee) \rightarrow  \fF_G \times^G V(\nu^\vee) \otimes \cO\left(\langle \lambda^*, \nu^\vee \rangle \cdot x \right)
  \end{align}
is regular on on all of $S \times X$

\end{itemize}

By choosing a faithful embedding $G \hookrightarrow	GL_n$, we can see that $\overline{\Y^\lambda}$ is a finite-type Noetherian scheme. Using the Cartan Decomposition, we see that $\overline{\Y^\lambda}(F) = \overline{\Gr^\lambda}(F)$ for any field $F$. So we see that $\overline{\Y^\lambda}$ is a possibly non-reduced thickening of $\overline{\Gr^\lambda}$. Our goal is to prove that $\overline{\Y^\lambda} = \overline{\Gr^\lambda}$, or equivalently, that $\overline{\Y^\lambda}$ is reduced.

\section{Generic smoothness}{\label{sec:tangent-space-computation}}

Let $ \lambda $ be a dominant coweight.  In this section, we will prove that $ \overline{\Y^\lambda} $ is generically reduced, by computing the tangent space at the point $ t^\lambda $.

Let $X$ scheme over $\CC$, and let $x$ be a closed point of $X$. Recall that the tangent space at $x$ can be identified with the space of maps $\Spec \CC[\varepsilon]/(\varepsilon^2) \rightarrow X$ that map the closed point of $\Spec \CC[\varepsilon]/(\varepsilon^2)$ to $x$.

Now, the tangent space to $t^\lambda$ in $ \Gr $ is $\mathfrak{g}(\cK)/(t^\lambda \mathfrak{g}(\cO) t^{-\lambda})$.  Choose a basis $ \{ X_r \} $ for $ \g $ consisting of root vectors and elements of the Cartan.  For each $ r $, let $ \alpha_r^\vee $ denote the weight of $ X_r $ (so $ \alpha_r^\vee $ is either a root or 0).  Thus we see that the tangent space to $ t^\lambda $ is $ \oplus_r t^{\langle \lambda, \alpha_r^\vee \rangle -1} \C[t^{-1}] X_r $.  For any tangent vector $ t^{\langle \lambda, \alpha_r^\vee \rangle -1} a X_r $, let us write $ (1 + \varepsilon t^{\langle \lambda, \alpha_r^\vee \rangle -1} a X_r)t^\lambda $ for the corresponding $ \Spec \CC[\varepsilon]/(\varepsilon^2) $ point of $ \Gr $.

\begin{Lemma}
If $(1 + \varepsilon t^{\langle \lambda, \alpha_r^\vee \rangle -1} a X_r)t^\lambda $ is a $  \Spec \CC[\varepsilon]/(\varepsilon^2) $ point of $ \overline{\Y^\lambda} $ then $ a \in \{a_0 + \dots + a_{k-1} t^{-(k-1)} \}$, where $ k = \langle \lambda, \alpha_r^\vee \rangle $.
\end{Lemma}

\begin{proof}
For any element $ a \in \C((t)) $, let us write $ \val(a) $ for the valuation of $ a $ (the most negative exponent occurring in $ a $).

Let $ \nu^\vee $ be a regular anti-dominant weight of $ G $ and consider a vector $ v_{\nu^\vee} $ of weight $ \nu^\vee $ lying in the representation $V(w_0 \nu^\vee) $ of highest weight $ w_0 \nu^\vee $.  

\noindent{\bf Case $\alpha_r^\vee \geq  0$:} 
We may assume that $ \langle X_r, \nu^\vee \rangle \neq 0 $ for all basis vectors $ X_r $ which lie in the Cartan. Fix a basis vector $ X_r $ such that $ \alpha_r^\vee $ is a positive root or 0.  Thus, since $ \nu^\vee$ is regular, $X_r v_{\nu^\vee} \ne 0 $.

Then for any $ a \in \C[t^{-1}] $ consider
\begin{align}
  \label{eq:3}
  (1 + \varepsilon t^{\langle \lambda, \alpha_r^\vee \rangle -1} a X_r)  t^{\lambda} v_{\nu^\vee} = t^{\langle \lambda, \nu^\vee \rangle}v_{\nu^\vee} + \varepsilon a t^{\langle  \lambda, \alpha_r^\vee + \nu^\vee \rangle -1} X_r v_{\nu^\vee} \in V(w_0\nu^\vee)((t)) \otimes \CC[\varepsilon]/(\varepsilon^2)
\end{align}
By the definition of $ \overline{\Y^\lambda}$, if $ (1 + \varepsilon t^{\langle \lambda, \alpha_r^\vee \rangle -1} a X_r)t^\lambda $ is a $  \Spec \CC[\varepsilon]/(\varepsilon^2) $ point of $ \overline{\Y^\lambda} $, then we know that the worst pole allowed in \eqref{eq:3} is $ - \langle \lambda, \nu^\vee \rangle $. The worst pole that occurs in \eqref{eq:3} is equal to $ -\val( a t^{\langle  \lambda, \alpha_r^\vee + \nu^\vee \rangle -1})$.  Thus we conclude that $ \val(a) > - \langle \lambda, \alpha_r^\vee \rangle$, and the result follows for this case.

\noindent{\bf Case $\alpha_r^\vee < 0$:}
Now we consider a basis vector $ X_r $ for which $ \alpha_r^\vee $ is a negative root.  Let $ E $ be a root vector corresponding the positive root $ -\alpha_r^\vee $.  Let $ w = E v_{\nu^\vee} $.  This is a non-zero vector of weight $ \nu^\vee -\alpha_r^\vee $ and $ X_r w \ne 0 $.  For any $ a \in \C[t^{-1}] $ consider
\begin{align}
  \label{eq:5}
(1 + \varepsilon t^{\langle \lambda, \alpha_r^\vee \rangle -1} a X_r)  t^{\lambda} w = t^{\langle \lambda, \nu^\vee - \alpha_r^\vee \rangle} w + \varepsilon a t^{\langle  \lambda, \alpha_r^\vee + \nu^\vee - \alpha_r^\vee \rangle -1} X_r w \in V(w_0\nu^\vee)((t)) \otimes \CC[\varepsilon]/(\varepsilon^2)
\end{align}
As before the worst pole allowed in \eqref{eq:5} is $ -\langle \lambda, \nu^\vee \rangle  $ and thus we conclude that $ a = 0 $. The result follows in this case too.
\end{proof}

\begin{Corollary} \label{cor:genreduced}
$\overline{\Y^\lambda} $ is smooth (and therefore reduced) along the open subset $ \Gr^\lambda \subset \overline{\Gr^\lambda}$. 
\end{Corollary}
\begin{proof}
From the above computation, we see that the dimension of the tangent space to $ \overline{\Y^\lambda} $ at $ t^\lambda $ is at most
$$
\sum_{\alpha^\vee}  \langle \lambda, \alpha^\vee \rangle
$$
where the sum ranges over all positive roots.  This number equals $ \langle \lambda, 2\rho^\vee \rangle $ which is the dimension of the smooth variety $ \Gr^\lambda $.  Thus we conclude that $ \overline{\Y^\lambda} $ is smooth at $ t^\lambda $.  By the $ G(\C[[t]]) $ action, we conclude that $ \overline{\Y^\lambda} $ is smooth along $ \Gr^\lambda $.

Finally, this implies that $\overline{\Y^\lambda}$ is reduced along $\Gr^\lambda$ via the general fact that a regular local ring is an integral domain (for example, \cite[Lemma 11.23]{AM}).
\end{proof}

\subsection{Reducedness and the Cohen-Macaulay property}

\begin{Theorem}{\label{thm:reduced-iff-cohenmacaulay}}
  The scheme $\overline{\Y^\lambda}$ is reduced if and only if it is Cohen-Macaulay. 
\end{Theorem}

\begin{proof}
It is known that the spherical Schubert varieties $\overline{\Gr^\lambda}$ are Cohen-Macaulay (for example, \cite[Theorem 8]{F}). Therefore if $\overline{\Y^\lambda}$ is reduced, then it is Cohen-Macaulay. 

We have shown in Corollary \ref{cor:genreduced}, that the schemes $\overline{\Y^\lambda}$ are generically reduced. For Cohen-Macaulay schemes, generic reduceness implies reducedness. Therefore, if $\overline{\Y^\lambda}$ is Cohen-Macaulay, then it is reduced.
\end{proof}

\section{The Beilinson-Drinfeld deformation of $ \overline{\Y^\lambda} $}
\label{sec:beilinson-drinfeld-deformation}
\subsection{A one-parameter family}

Let $X$ be a smooth curve, and let $p \in X$ be a  closed point. Let $\Gr_{G,X\times p}$ be the closed sub ind-scheme of the Beilinson-Drinfeld Grassmannian for two points on a curve where the second point is fixed at $p$. Explicitly, an $S$-point of $\Gr_{G,X\times p}$ consists of the following data:
\begin{itemize}

\item $x: S \rightarrow X$. Let $\Gamma_x$ denote the graph of $x$. Let $\Gamma_p = S \times \{p\} \subset S \times X $ be the graph of the constant map taking value $p$.

\item$\fF_G$ a $G$-bundle on $S \times X$.

\item A trivialization $\phi : \fF^0_G \dashedrightarrow \fF_G$ defined on $S\times X - (\Gamma_x \cup \Gamma_p)$

\end{itemize}

\noindent Let $\pi_x : \Gr_{G,X\times p} \rightarrow X$ be the map that remembers the point $x$. Then if $y \in X$ and $y \neq p$, then we can canonically identify $\pi_x^{-1}(y)$ with $\Gr_{X,y} \times \Gr_{X,p}$. When $y = p$, we have $\pi_x^{-1}(p) = \Gr_{X,p}$.

Let $\lambda, \mu$ be two  dominant coweights. Now let us consider the following closed sub-ind-scheme $\overline{\Y^{\lambda,\mu}_{X\times p}}$ of $\Gr_{G,X\times p}$. An $S$-point of $\overline{\Y^{\lambda,\mu}_{X\times p}}$ consists of an $S$-point of $\Gr_{G,X\times p}$ subject to the following condition:

\begin{itemize}

\item For every dominant weight $\nu^\vee$, the composed map
  \begin{align*}
\phi_{\nu^\vee} : \fF^0_G \times^G V(\nu^\vee) \dashedrightarrow \fF_G \times^G V(\nu^\vee) \rightarrow  \fF_G \times^G V(\nu^\vee) \otimes \cO\left(\langle \lambda^*, \nu^\vee \rangle \cdot \Gamma_x + \langle \mu^*, \nu^\vee\rangle \cdot \Gamma_p \right)
  \end{align*}
is regular on on all of $S \times X$
\end{itemize}

By embedding $G$ into some $GL_n$, we can see that  $\overline{\Y^{\lambda,\mu}_{X\times p}}$ is a finite-type scheme.

\begin{Proposition}{\label{prop-reduced-if}}
Suppose $\lambda$ and $\mu$ are dominant coweights.
If $\overline{\Y^{\lambda + \mu}}$ is reduced, then both $\overline{\Y^{\lambda}}$ and $\overline{\Y^{\mu}}$ are reduced.
\end{Proposition}

\begin{proof}
First we will prove that the map $\pi_x$ is in fact flat. Over $X-\{p\}$, the map $\pi_x$ is a fiber bundle and is therefore flat. Let us consider the base change of the family $\pi_x$ over $\Spec \cO_{X,x}$ where $\cO_{X,x}$ is the local ring at $x$. By \cite[Proposition 14.16]{GW} a locally Noetherian family over a DVR with reduced special fiber such that the entire family is the \emph{set-theoretic} closure of the generic fiber is automatically flat. The \emph{set-theoretic} closure condition is clear (see also \cite[Proposition 2.1.4]{Zhu}).

Then because $\pi_x$ is proper and flat, by \cite[Proposition 12.2.4(v)]{EGAIV3}, we know that the locus of $y \in X$, such that the scheme-theoretic fiber $\pi_{x}^{-1}(y)$ is geometrically reduced is open.
\end{proof}

\subsection{Some remarks on the family $\pi_x$.}

The proof of Proposition \ref{prop-reduced-if} tells us that the reducedness of the special fiber of $\pi_x$ implies that $\pi_x$ is flat and that the generic fiber is reduced. Here we will show that the converse is true. This fact will not be used in the sequel.

\begin{Proposition}{\label{prop-reduced-onlyif}}
 Suppose that $\pi_x : \overline{\Y^{\lambda,\mu}_{X\times p}} \rightarrow X $ is flat and that
 both $\overline{\Y^{\lambda}}$ and $\overline{\Y^{\mu}}$ are reduced. Then $\overline{\Y^{\lambda + \mu}}$ is reduced.
\end{Proposition}

This argument is very similar to the argument used to prove \cite[Proposition 2.1.4]{Zhu}.

\begin{proof}
Consider the determinant line bundle $\cL$, and let $\mathcal{J}$ be the nilradical sheaf on $\overline{\Y^{\lambda + \mu}}$. We get a short exact sequence
\begin{align}
0 \rightarrow \cJ \rightarrow \cO_{\overline{\Y^{\lambda + \mu}}} \rightarrow i_* \cO_{\overline{\Gr^{\lambda + \mu}}} \rightarrow 0
\end{align}
where $i : \overline{\Gr^{\lambda + \mu}} \rightarrow \overline{\Y^{\lambda + \mu}}$ is the natural map.

For $k >> 0$, we can tensor with $\cL^{\otimes k}$ and take global sections to obtain the following short exact sequence
\begin{align}
0 \rightarrow H^0(\overline{\Y^{\lambda + \mu}}, \cJ \otimes \cL^{\otimes k} ) \rightarrow H^0(\overline{\Y^{\lambda + \mu}}, \cL^{\otimes k}) \rightarrow H^0(\overline{\Gr^{\lambda + \mu}}, \cL^{\otimes k}) \rightarrow 0
\end{align}

By flatness of $\pi_x$, for $k >> 0$, we have an isomorphism $H^0(\overline{\Y^{\lambda + \mu}}, \cL^{\otimes k}) \cong H^0(\overline{\Y^{\lambda}} \times \overline{\Y^{\mu}}, \cL^{\otimes k})= H^0(\overline{\Gr^{\lambda}} \times \overline{\Gr^{\mu}}, \cL^{\otimes k}) \cong H^0(\overline{\Gr^{\lambda + \mu}}, \cL^{\otimes k})$, where the last isomorphism is given by the tensor product property of affine Demazure modules appearing in integrable vacuum representations \cite[Lemma 2.1.4]{Zhu}.

 Therefore, we have $H^0(\overline{\Y^{\lambda + \mu}}, \cJ \otimes \cL^{\otimes k} ) = 0$ for $k$ sufficiently large. As $\cL$ is ample, we conclude that $\cJ =0$.
\end{proof}

\begin{Remark}
If we knew that the map $\pi_x : \overline{\Y^{\lambda,\mu}_{X\times p}} \rightarrow X $ was flat for all $\lambda$ and $\mu$, then the problem of proving the reducedness of $\overline{\Y^{\nu}}$ would greatly simplify. For example, by taking $G$ to be of adjoint type it would suffice to check the reducedness for all fundamental coweights. When $G=PGL_n$, all fundamental coweights $\varpi$ are miniscule, and $\overline{\Y^\varpi}$ is in fact smooth (and therefore reduced) as we prove in section \ref{sec:tangent-space-computation}.
Alternatively, using Proposition \ref{prop-reduced-if} and \ref{prop-reduced-onlyif} it would also suffice to prove that $\overline{\Y^{\nu}}$ is reduced for a single regular dominant coweight $\nu$.

Therefore, we see that the question of reducedness of $\overline{\Y^{\nu}}$ is very closely related to the question of the flatness of $\pi_x : \overline{\Y^{\lambda,\mu}_{X\times p}} \rightarrow X $ for all $\lambda$ and $\mu$.  Unfortunately, it does not seem easy to directly prove the flatness of $ \pi_x $.
\end{Remark}

\section{Some simplifying arguments}
\label{sec:Some simplifying arguments}

\subsection{Reduction to big cell}

The big cell in the affine Grassmannian is defined as $ U = G(\C[t^{-1}]) G(\C[[t]]) / G(\C[[t]]) \subset \Gr $.  Define $ G_1[t^{-1}] $ be the kernel of the evaluation morphism $ G\big(\C[t^{-1}]\big) \rightarrow G $.  Recall that the natural map $ G_1[t^{-1}]\rightarrow U $ is an isomorphism.

\begin{Proposition} \label{prop:bigcellreduced}
Let $U \subset \Gr$ be the big cell in the affine Grassmannian. Then $\overline{\Y^\lambda}$ is reduced if and only if $U \cap \overline{\Y^\lambda}$ is reduced.
\end{Proposition}
\begin{proof}
One direction is clear. For the other direction, suppose that $U \cap \overline{\Y^\lambda}$ is reduced. Note that $\overline{\Y^\lambda}$ has a $G(\C[[t]])$-action, and each orbit has non-empty intersection with  $U \cap \overline{\Y^\lambda}$. Using this action, we see that every closed point of $\overline{\Y^\lambda}$ has a open neighborhood that is reduced.
\end{proof}

\subsection{Intersecting Orbit Closures}

\begin{Definition}
Let $\lambda$ and $\mu$ be dominant coweights whose difference lies in the coroot lattice. Write $\lambda = \sum_i a_i \alpha_i$, and $\mu = \sum_i b_i \alpha_i$. Then define the {\bf meet} $\lambda \wedge \mu = \sum_i \min \{ a_i, b_i\} \alpha_i$.
\end{Definition}
The following lemma follows from \cite[Theorem 1.3]{Stembridge}
\begin{Lemma}
\label{lem-meet-of-dom-weights}
For $\lambda, \mu$ dominant, $\lambda\wedge \mu$ is dominant.
\end{Lemma}

\begin{Proposition}
  \label{prop:intersecting-schubert-varieties}
Suppose $\lambda$ and $\mu$ are dominant coweights whose difference lies in the coroot lattice.
The scheme theoretic intersection $\overline{\Gr^\lambda} \cap \overline{\Gr^\mu}$ is equal to $\overline{\Gr^{\lambda \wedge \mu}}$
\end{Proposition}

\begin{proof}
On the level of sets, this follows from Lemma \ref{lem-meet-of-dom-weights}. The main content is that the scheme-theoretic intersection is reduced. The reducedness follows because Schubert varieties in a Kac-Moody partial flag variety are Frobenius split compatibly with their Schubert subvarieties \cite[Proposition 5.3: Assertion I]{KumSch}.
\end{proof}

The corresponding fact is obvious for the modular versions of these spaces:
\begin{Proposition}
  \label{prop:intersecting-moduli-version-schubert-varieties}
Scheme-theoretically, we have $\overline{\Y^\lambda} \cap \overline{\Y^\mu} = \overline{\Y^{\lambda \wedge \mu}}$
\end{Proposition}

\section{Type A calculations}
\label{sec:Type A calculations}

Now let us focus on $G = SL_n$ and dominant coweights of the form $k n \varpi_1$, where $k$ is some positive integer. As dimension is insensitive to nilpotents, we have $ \dim \overline{\Y^{kn \varpi_1}} = \dim \overline{\Gr^{kn \varpi_1}} = kn(n-1)$.

Consider the intersection $\overline{\Y^{k n\varpi_1}} \cap U$, where as before $U \subset \Gr$ is the big cell in the affine Grassmannian. Because $U$ is an ind-affine scheme, the intersection $\overline{\Y^{kn\varpi_1}} \cap U$ is an ordinary affine scheme. Let $A$ be the ring of regular functions on $\overline{\Y^{kn\varpi_1}} \cap U$.

We have an identification $ U = G_1[t^{-1}] $.  Since $ G = SL_n \hookrightarrow M_n$, we can embed $ U \hookrightarrow I_n + t^{-1}M_n([t^{-1}]) $ where $ M_n $ denotes the variety of $ n \times n $ matrices.  Thus we can regard $ \overline{\Y^{kn\varpi_1}} \cap U $ as a subscheme of $ I_n + t^{-1}M_n([t^{-1}]) $.  Applying the case $ \nu^\vee = \varpi_1^\vee $ in the definition of $ \overline{\Y^{kn\varpi_1}}$, we see that $ \overline{\Y^{kn\varpi_1}} \cap U $ is a subscheme of the $kn^2 $ dimensional affine space $ I_n + t^{-1}M_n([t^{-1}])_{\le k} $ consisting of matrices whose entries are polynomials in $t^{-1} $ of degree at most $ k $ (and which evaluate to $I_n$ at $ t^{-1} = 0 $).

Let us consider variables $x_{i,j}^{(s)}$ where the indices $i$ and $j$ vary in the set $\{ 1, \ldots, n\}$, and the indices $s$ vary in the set $\{1, \ldots, k \}$. For fixed $s$, let $X^{(s)}$ be the $n \times n$ matrix whose entries are $x_{i,j}^{(s)}$. Consider matrix valued polynomials of the following form:
\begin{align}
  \label{eq:1}
 X = I_n + \sum_{s = 1}^{k} X^{(s)} \cdot t^{-s}
\end{align}
Let us also write
\begin{align}
  \label{eq:2}
  \det(X) = \sum_{r=1}^{kn} \det^{(r)} \cdot t^{-r}
\end{align}
which defines $\det^{(r)}$ as certain polynomials in the variables $x_{i,j}^{(s)}$.

\begin{Proposition}
The ring $A$ is isomorphic to $k[x_{i,j}^{(s)}]$ modulo the ideal generated by the polynomials $\mathrm{det}^{(r)}$ where $r \in \{1, \ldots, kn\}$.
\end{Proposition}

\begin{proof}
The equations $ \det^{(r)} $ define the intersection $ U \cap \big(I_n + t^{-1}M_n([t^{-1}])_{\le k}\big)$.  By the above analysis, this scheme contains $ \overline{\Y^{kn\varpi_1}} \cap U $.  To see that there are no additional equations, we note that it suffices to check the condition in the definition of $ \overline{\Y^{kn\varpi_1}} $ for $ \nu^\vee =\varpi_\ell^\vee$ a fundamental weight.  This amounts to show that all $ \ell \times \ell $ minors of $ X $ have degree (in $t^{-1}$) at most $ kl = -\langle kn\varpi_1^*, \varpi_\ell^\vee \rangle $, for $\ell = 2, \dots, n-1 $.  But this is immediate from the definition.
\end{proof}

In particular, we see that $\Spec A$ is cut out by $kn$ equations in an affine space of dimension $kn^2$. Because $\overline{\Y^{kn\varpi_1}}$ has dimension $kn(n-1)$, we see that $\Spec A$ is a complete intersection inside this affine space. In particular, it is a Cohen-Macaulay scheme. As in the proof of Proposition \ref{prop:bigcellreduced}, the $ G(\C[[t]]) $ action proves that $\overline{\Y^{kn\varpi_1}}$ is also Cohen-Macauley.  Furthermore, applying the unique non-trivial diagram automorphism of $SL_n$, we obtain that $\overline{\Y^{kn \varpi_{n-1}}}$ is also Cohen-Macaulay.
Applying Theorem \ref{thm:reduced-iff-cohenmacaulay}, we have the following.

\begin{Corollary}
For $G = SL_n$ and any integer $k \geq 1$, both $\overline{\Y^{kn\varpi_1}}$ and $\overline{\Y^{kn\varpi_{n-1}}}$ are reduced.
\end{Corollary}

We learned of this method of proving reducedness of a scheme from a similar argument in \cite[{\it Proofs of Theorems 2 and 3} ]{Knu}.
 
\section{Generating more cases}
\label{section: Generating more cases}

We use two facts to prove the reducedness of many more cases. These two facts are:

\begin{enumerate}
\item The scheme-theoretic intersection of Schubert varieties is reduced, and moreover $\overline{\Gr^\lambda} \cap \overline{\Gr^\mu} = \overline{\Gr^{\lambda\wedge\mu}}$ by Proposition \ref{prop:intersecting-schubert-varieties}.  By Proposition \ref{prop:intersecting-moduli-version-schubert-varieties}, we have $\overline{\Y^\lambda} \cap \overline{\Y^\mu} = \overline{\Y^{\lambda \meet \mu}}$

\item If $\overline{\Y^{\mu_1+\mu_2}}$ is reduced for $\mu_1$ and $\mu_2$ two dominant coweights, then $\overline{\Y^{\mu_1}}$ and $\overline{\Y^{\mu_2}}$ are both reduced by Proposition \ref{prop-reduced-if}. Given a dominant coweight $\lambda$, we say that a dominant coweight $\mu$ is a \emph{summand} of $\lambda$ if $\lambda-\mu$ is dominant.
\end{enumerate}

Thus we see that if we know reducedness for any set of dominant coweights, then we know it also for the set all dominant coweights obtained by repeatedly taking summands and meets.

Let $\mathcal{X}$ be the smallest set of dominant coweights for $SL_n$ that contains $\{ k n \varpi_1 : k >0 \}$ and $\{ k n \varpi_{n-1}: k >0 \}$ and is closed under summands and meets. Then we know the truth of the reducedness conjecture for all elements of $\mathcal{X}$.

\begin{Proposition}
  The set $\mathcal{X}$ consists of all dominant coweights that can be written in the form $a \varpi_i + b \varpi_{i+1}$ for $a, b \geq 0$ and $i \in \{ 1, \ldots, {n-2} \}$.
\end{Proposition}

\begin{proof}
Consider the set $\mathcal{T}$ of piecewise linear functions $f: [0,n]\rightarrow \R_{\geq 0}$ whose graphs are triangles: $f$ consists of the straight lines from $(0,0)$ to $(a,b)$ to $(n,0)$, for some $a\in [1,n-1]$ and $b\geq 0$ (see Figure \ref{figure: triangles}).  Observe that for $f_1, f_2 \in \mathcal{T}$ we have $\min\{ f_1, f_2\} \in \mathcal{T}$.

Define a  ``discrete sampling'' map $ \pi : \mathcal{T} \rightarrow \mathfrak{h}_\R$ by $f \mapsto \sum_{i=1}^{n-1} f(i) \alpha_i$.  Note that
\begin{enumerate}
\item $ \pi( \min\{f_1,f_2\}) = \pi(f_1) \wedge \pi(f_2) $
\item $\varpi_i$ is the image of the element with apex at $(i, \tfrac{i(n-i)}{n})$
\item If the apex $(a,b)$ of $f$ has $i\leq a\leq i+1$, then $\pi(f)\in \R \varpi_i \oplus \R \varpi_{i+1}$.
\end{enumerate}
Now suppose we take two coweights $$\lambda_1, \lambda_2 \in\{a \varpi_i + b\varpi_{i+1} : i\in I, \ a,b\geq 0\}$$  Then both are in the image of $\mathcal{T}$, and by properties (1) and (3) so is $\lambda_1\wedge\lambda_2$.  Hence the set of coweights of this form is closed under meets.  It is clearly closed under summands.

On the other hand, for each $i$ and $N\geq 0$ there is some $a \varpi_i + b \varpi_{i+1}\in \mathcal{X}$ with $a, b \geq N$.  Indeed, by appropriately choosing $\ell, \ell'$, such elements can be produced of the form $\ell \varpi_1 \wedge \ell' \varpi_{n-1}$.  Since $\mathcal{X}$ is closed under taking summands, the claim follows.
\end{proof}

\begin{figure}
\begin{tikzpicture}[scale=1]
\draw[->] (-0.5,0) -- (6,0);
\draw[->] (0,-0.5) -- (0,3);

\draw[-] (0,0) -- (4,2);
\draw[-] (4,2) -- (5,0);
\draw node at (5,-0.3) {\small $(n,0)$};
\draw node at (4,2.3) {\small $(a,b)$};
\end{tikzpicture}
\caption{\label{figure: triangles}
The graph of the element of $\mathcal{T}$ with apex at $(a,b)$.
}
\end{figure}
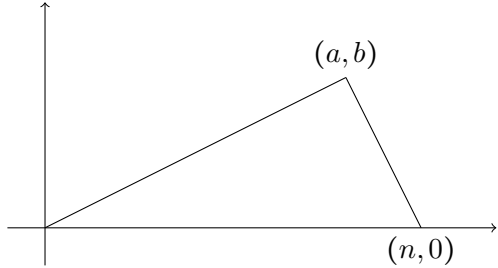

\begin{Theorem} \label{theorem: main theorem}
For $G= SL_n$ and all $\lambda \in \mathcal{X}$, the scheme $\overline{\Y^\lambda}$ is reduced.  In particular, for $\lambda \in \mathcal{X}$ we have $\overline{\Y^\lambda} = \overline{\Gr^\lambda}$.
\end{Theorem}

\begin{Remark}
  For $SL_2$ and $SL_3$ this is everything.
\end{Remark}

\section{Ideal generators for affine Grassmannian slices}
\label{section: Ideal generators for Lusztig slices}

Our goal is now to study the conjecture \cite[Conjecture 2.20]{KWWY} describing the ideal generators for affine Grassmannian slices.  We will give a positive answer in those cases covered by Theorem \ref{theorem: main theorem}.

\begin{Remark}
To follow the notation of \cite{KWWY}, we will work now in the thick affine Grassmannian $G((t^{-1}))/G[t]$.  This does not affect results since we are studying subschemes of $\overline{\Y^\lambda}$, a finite-type scheme naturally embedded in both the thick and the thin affine Grassmannian.
\end{Remark}

Let $ \mu $ be a dominant coweight with $ \mu \le \lambda$.   Let $ \Gr_\mu = \Gm t^{w_0 \mu} $ and let $ \Gr^\lambda_\mu = \overline{\Gr^\lambda} \cap \Gr_\mu $.  This is a transverse slice to $ \overline{\Gr^\lambda} $ at the point $ t^{w_0 \mu} $.   It is known by \cite[Proposition 5.3]{KumSch} that $ \Gr^\lambda_\mu $ is reduced.  Also let $\Y_\mu^\lambda = \overline{\Y^\lambda} \cap \Gr_\mu$.

Recall from \cite[Section 3]{KWWY}, that $ \Gm $ is a Poisson algebraic group, that  $ \Gr_\mu $ is a Poisson homogeneous space for $ \Gm $, and that both $\Y_\mu^\lambda$ and $ \Gr^\lambda_\mu $ are Poisson subschemes of $ \Gr_\mu$.

Let $ V $ be a representation of $ G $, let $ v \in V, \beta \in V^* $.  Then we have a matrix coefficient $ \Delta_{\beta, v} \in \O(G) $.  The group $G_1[[t^{-1}]]$ acts on $V[[t^{-1}]]$, and for $ s \in \N $ we define $ \Delta_{\beta, v}^{(s)} \in \O(\Gm) $ by
$$ \Delta_{\beta,v}(g) = \sum_{s \geq 0 }\Delta_{\beta,v}^{(s)}(g) t^{-s} $$
It is useful to encode these functions in the form of formal series, and we will denote
$$ \Delta_{\beta,\gamma}(u) := \sum_{s\geq 0} \Delta_{\beta,\gamma}^{(s)} u^{-s} \ \in \ \O\left(G_1[[t^{-1}]]\right) [[ u^{-1}]] $$
In particular, the Poisson bracket on $\O(G_1[[t^{-1}]])$ is given by
\begin{equation} \label{eq: poisson bracket of minors}
(u-v) \left\{ \Delta_{\beta_1,\gamma_1}(u), \Delta_{\beta_2,\gamma_2}(v) \right\} = \sum_a \Big( \Delta_{\beta_1, J_a \gamma_1}(u) \Delta_{\beta_2, J^a \gamma_2}(v) - \Delta_{J_a \beta_1, \gamma_1}(u) \Delta_{J^a \beta_2, \gamma_2}(v) \Big)
\end{equation}
where $\{J_a\}$, $\{J^a\}$ are dual bases for $\g$ with respect to the bilinear form $(\cdot,\cdot)$, which in practice we will take to be (see Section \ref{Notation})
$$ \{ h_i, e_{\alpha^\vee}, f_{\alpha^\vee} : i\in I, \alpha^\vee \in \Phi_+^\vee\}, \ \ \ \{h^i, d_{\alpha^\vee} f_{\alpha^\vee}, d_{\alpha^\vee} e_{\alpha^\vee} : i\in I, \alpha^\vee \in \Phi_+^\vee \} $$
We follow here the convention of \cite{KWWY}: this Poisson structure corresponds to the $r$-matrix $\Omega/(u-v)$ on $\operatorname{Lie}\big(\Gm\big) = t^{-1} \g[[t^{-1}]]$, where $\Omega$ is the Casimir 2-tensor for $(\cdot,\cdot)$.

Consider the irreducible fundamental representation $V(\varpi_i^\vee)$ for $\g$.  Fix a highest weight vector $v_i\in V(\varpi_i^\vee)$, and a lowest weight dual vector $v_i^\ast \in V(\varpi_i^\vee)^\ast$.  Define
$$ f_i(u)  = \sum_{s>0}f_i^{(s)} u^{-s} :=  d_i^{\phantom{.}-1/2} \frac{\Delta_{v_i^\ast,f_i v_i}(u)}{\Delta_{v_i^\ast, v_i}(u)} $$
As explained in \cite[Section 2G]{KWWY}, the matrix coefficients of $V(\varpi_i^\vee)$ are well-defined on $G_1[[t^{-1}]]$ even if $G$ is not simply-connected

\begin{Theorem}[Theorem 3.12 in \cite{KWWY}] \label{subalgebra}
The subalgebra $\O(\Gr_\mu) \subset \O(\Gm)$ is Poisson generated by the elements
$$ \Delta_{e_i v_i^\ast, v_i}^{(s)}  \text{ and } \Delta_{v_i^\ast, v_i}^{(s)} \text{ for all } i\in I, s>0,$$
$$ \text{ and } f_i^{(s)} \text{ for all } i\in I, s>\mu_i$$
where $ \mu_i = \langle \mu^*, \alpha_i^\vee \rangle $.
\end{Theorem}

Let $ J_\mu^\lambda $ denote the Poisson ideal of $ \O(\Gr_\mu) $ which is generated by $\Delta_{v_i^\ast,v_i}^{(s)}$ for $i\in I$ and $s>m_i$, where $ \lambda - \mu = \sum_{i \in I} m_i \alpha_{i^*} $.  In \cite[Proposition 2.21]{KWWY}, the authors proved that the set-theoretic vanishing locus of $ J_\mu^\lambda $ is $ \Gr^\lambda_\mu $ and they conjectured the following:

\begin{Conjecture} \label{conj: KWWY conj}
$J_\mu^\lambda $ is the ideal of $ \Gr^\lambda_\mu $.
\end{Conjecture}

In the following section we will prove:
\begin{Theorem} \label{th:reducedness}
$J_\mu^\lambda $ is the ideal of $ \Y^\lambda_\mu $.
\end{Theorem}

\begin{Corollary}
Conjecture \ref{conj: KWWY conj} is true when $G = SL_n$ and $\lambda \in \mathcal{X}$ (as defined in Section \ref{section: Generating more cases}).
\end{Corollary}

Consider the (ordinary) ideal $I_\mu^\lambda \subset \O(\Gm)$ generated by
$$J_\mu^\lambda \subset \O(\Gr_\mu) \subset \O(\Gm)$$
We will study $ J^\lambda_\mu $ using $ I^\lambda_\mu $. This is possible because of the following simple result concerning the map $ p : \Gm \rightarrow \Gr_\mu $ defined by $ g \mapsto g t^{w_0\mu} $.

\begin{Proposition}
$ J_\mu^\lambda $ is the ideal of $ \Gr^\lambda_\mu $ as a subvariety of $ \Gr_\mu $ if and only if $ I^\lambda_\mu $ is the ideal of $ p^{-1}(\Gr^\lambda_\mu)  $ as a subvariety of $ \Gm$.
\end{Proposition}

\begin{proof}
Let $ K^\lambda_\mu $ denote the ideal of $ \Gr^\lambda_\mu$.  From the definition of pullback, it follows that the ideal of $ p^{-1}(\Gr^\lambda_\mu) $ is the ideal generated by $ K^\lambda_\mu $ in $ \O(\Gm) $.  Thus, we see that if $ J^\lambda_\mu = K^\lambda_\mu $, then $ I^\lambda_\mu $ is the ideal of $ p^{-1}(\Gr^\lambda_\mu)  $.

Conversely, suppose that $ I^\lambda_\mu $ is the ideal of $ p^{-1}(\Gr^\lambda_\mu) $.  Thus, both the scheme $ \Gr^\lambda_\mu $ and the subscheme of $ \Gr_\mu $ defined by $ J^\lambda_\mu $ pullback under $ p$ to the same subscheme of $ \Gm $.  But, since $ \Gm \rightarrow \Gr_\mu $ is a locally trivial fibration, this implies that these two subschemes of $ \Gr_\mu $ are equal.  So the result follows.
\end{proof}

Thus in order to prove Theorem \ref{th:reducedness}, it suffices to prove the following result.
\begin{Theorem} \label{th:idealofY}
$I^\lambda_\mu $ is the ideal of $ m^{-1}(\overline{\Y^\lambda})$.
\end{Theorem}
Here we write $ m : \Gm \rightarrow \Gr $ for the morphism given by $ p:\Gm \rightarrow \Gr_\mu $ followed by the inclusion of $ \Gr_\mu $ into $ \Gr $.

\section{Poisson brackets}
\label{sec: Poisson brackets}

We begin with a general fact about ideals generated by Poisson ideals:
\begin{Lemma}
Suppose $A \subset B$ is an inclusion of Poisson algebras, and that $J\subset A$ is a Poisson ideal.  Consider the ordinary ideal $I = JB \subset B$.  Then $\{ A, I\} \subset I$.
\end{Lemma}
\begin{Corollary} \label{Poisson brackets}
$I_\mu^\lambda \subset \O(\Gm)$ is closed under taking Poisson brackets with $\O(\Gr_\mu)$.
\end{Corollary}

Recall a well-known ``delta-function'' property of the  formal series
$$ \frac{u^{-1}}{1-u^{-1} v} = u^{-1} + u^{-2} v + u^{-3} v^2 + \ldots $$
Namely, for any $T(u) = \sum_{n\in \ZZ} T_n u^{-n} \in Y((u^{-1}))$ we have
\begin{equation} \label{eq: residue}
\res_u \left(\frac{u^{-1}}{1-u^{-1} v} T(u)\right)  = T(v)_+ \in Y[v]
\end{equation}
where we denote $T(v)_+ := \sum_{n\geq 0} T_n v^n$, and where $\res_u$ denotes the formal residue at $u=0$ (i.e. the coefficient of $u^{-1}$).

\begin{Lemma} \label{lemma: bracket with f_i}
Let $j \in I$ and $k \geq 1$.  There exist polynomials
$$p, q_{\alpha^\vee} \in \O\left(\Gm\right)[v], \qquad \forall {\alpha^\vee} \in \Phi_+^\vee,$$
of degree $k-1$ in $v$, such that for all $i\in I$ and weight vectors $\gamma \in V(\varpi_i)$ we have
$$ \{f_j^{(k+1)}, \Delta_{v_i^{\ast},\gamma}(v)\} = v^k  d_j^{1/2} \Delta_{v_i^{\ast}, f_j \gamma}(v) + ( \alpha_j^\vee, \operatorname{wt}(\gamma) ) p(v) \Delta_{v_i^{\ast}, \gamma}(v) + \sum_{\alpha^\vee \in \Phi_+^\vee} q_{\alpha^\vee}(v) \Delta_{v_i^{\ast}, e_{\alpha^\vee} \gamma}(v)  $$
\end{Lemma}
\begin{proof}
Using the definition of $f_j(u)$ and the formula (\ref{eq: poisson bracket of minors}) for the Poisson bracket, as well as the identity $\{a^{-1}, b\} = -a^{-1} \{ a, b\} a^{-1}$ valid in any Poisson algebra, one can show that
\begin{align} 
	(u-v)\{f_j(u), \Delta_{v_i^\ast,\gamma}(v)\} &= d_j^{1/2} \Delta_{v_i^\ast, f_j \gamma}(v) - ( \alpha_j^\vee, \operatorname{wt}(\gamma) ) f_j(u) \Delta_{v_i^\ast, \gamma}(v) \nonumber \\
&\phantom{X}+  \sum_{\alpha^\vee\in \Phi_+^\vee} d_j^{-1/2}d_{\alpha^\vee} \left( \frac{\Delta_{v_j^\ast, f_{\alpha^\vee} f_j v_j}(u)}{\Delta_{v_j^\ast, v_j}(u)}- \frac{\Delta_{v_j^\ast, f_j v_j}(u) \Delta_{v_j^\ast, f_{\alpha^\vee} v_j}(u)}{\Delta_{v_j^\ast, v_j}(u)^2} \right) \Delta_{v_i^\ast, e_{\alpha^\vee} \gamma}(v) \label{eq: bracket with f}
\end{align}
where all series are expanded as Laurent series in $u^{-1}$ and $v^{-1}$.

Now observe that
$$	\{f_j^{(k+1)}, \Delta_{v_i^\ast,\gamma}(v) \} = \ \res_u \Big( u^k \{ f_j(u), \Delta_{v_i^\ast, \gamma}(v) \} \Big) $$
We will rewrite the right-hand side.  First we use equation (\ref{eq: bracket with f}) with both sides multiplied by $u^k \frac{u^{-1}}{1-u^{-1} v}$, and then we apply the identity (\ref{eq: residue}):
\begin{align*}
& \res_u \left( u^k \frac{u^{-1}}{1-u^{-1}v} \bigg[ d_j^{1/2} \Delta_{v_i^\ast, f_j \gamma}(v) - ( \alpha_j^\vee, \operatorname{wt}(\gamma) ) f_j(u) \Delta_{v_i^\ast, \gamma}(v) \right.  \\  & \left.+ \sum_{\alpha^\vee}  d_j^{-1/2}d_{\alpha^\vee}  \Big( \frac{\Delta_{v_j^\ast, f_{\alpha^\vee} f_j v_j}(u)}{\Delta_{v_j^\ast, v_j}(u)}- \frac{\Delta_{v_j^\ast, f_j v_j}(u) \Delta_{v_j^\ast, f_{\alpha^\vee} v_j}(u)}{\Delta_{v_j^\ast, v_j}(u)^2} \Big) \Delta_{v_i^\ast, e_{\alpha^\vee} \gamma}(v) \bigg] \right)  \\
=& \ v^k d_j^{1/2} \Delta_{v_i^\ast, f_j \gamma}(v) - \Big[ v^k ( \alpha_j^\vee, \operatorname{wt}(\gamma) ) f_j(v) \Big]_+ \Delta_{v_i^\ast, \gamma}(v) \\
& + \sum_{\alpha^\vee} \left[v^k d_j^{-1/2}d_{\alpha^\vee}  \Big( \frac{\Delta_{v_j^\ast, f_{\alpha^\vee} f_j v_j}(v)}{\Delta_{v_j^\ast, v_j}(v)}- \frac{\Delta_{v_j^\ast, f_j v_j}(v) \Delta_{v_j^\ast, f_{\alpha^\vee} v_j}(v)}{\Delta_{v_j^\ast, v_j}(v)^2} \Big)\right]_+ \Delta_{v_i^\ast, e_{\alpha^\vee} \gamma}(v)
\end{align*}
The claim now follows, with
\begin{align*}
p(v) & = - \left[v^k f_j(v)\right]_+, \\
q_{\alpha^\vee}(v) & = \left[v^k d_j^{-1/2}d_{\alpha^\vee}  \Big( \frac{\Delta_{v_j^\ast, f_{\alpha^\vee} f_j v_j}(v)}{\Delta_{v_j^\ast, v_j}(v)}- \frac{\Delta_{v_j^\ast, f_j v_j}(v) \Delta_{v_j^\ast, f_{\alpha^\vee} v_j}(v)}{\Delta_{v_j^\ast, v_j}(v)^2} \Big)\right]_+
\end{align*}
\end{proof}

\begin{Proposition}
$I_\mu^\lambda$ is generated as an ordinary ideal by
$$\Delta_{\beta, \gamma}^{(s)}\ \ \text{ for } s > m_i + \langle\mu^*, \varpi_i^\vee - \operatorname{wt} (\gamma) \rangle $$
over all $i\in I$,  where $\beta, \gamma$ range over weight bases for $V(\varpi_i^\vee)^\ast$ and $V(\varpi_i^\vee)$, respectively.
\end{Proposition}

In \cite[Proposition 2.15]{KWWY}, we proved this for the case $ \mu = 0 $.  The current proof follows the same strategy, making use of the previous lemma.
\begin{proof}
Denote the ideal generated by these elements by $\widetilde{I}_\mu^\lambda$.  To begin we show that $\widetilde{I}_\mu^\lambda \subset I_\mu^\lambda$.

 We first prove this claim for elements of the form $\Delta_{v_i^\ast, \gamma}^{(s)}$, proceeding by downward induction on the weight of $\gamma$ (which we may assume is a weight vector).  The base case $\Delta_{v_i^\ast, v_i}^{(s)}$ follows from the definition of $J_\mu^\lambda$.  Now suppose that $\gamma \in V(\varpi_i^\vee)$ is not highest weight, so that $ \gamma = \sum f_j \gamma_j$ for some weight vectors $\gamma_j$ of higher weight than $\gamma$.

By the inductive assumption, $\Delta_{v_i^\ast, \gamma_j}^{(s)}$ and $\Delta_{v_i^\ast, e_{\alpha^\vee} \gamma_j}^{(s)}$ are in $I_\mu^\lambda$ for $s> m_i + \langle \mu^*, \varpi_i^\vee - \text{wt}(\gamma_j)\rangle$. In this case, Lemma \ref{Poisson brackets} implies that $\{ f_j^{(\mu_j+1)}, \Delta_{v_i^\ast, \gamma_j}^{(s)}\} \in I_\mu^\lambda$, as $f_j^{(\mu_j+1)} \in \O(\Gr_\mu)$ by Theorem \ref{subalgebra}.

From the $k = \mu_j$ case of the previous lemma,
$$ v^{\mu_j} d_j^{1/2} \Delta_{v_i^\ast, f_j \gamma_j}(v) = \{ f_j^{(\mu_j+1)}, \Delta_{v_i^\ast, \gamma_j}(v)\} -  (\alpha_j^\vee, \text{wt}(\gamma_j)) p(v) \Delta_{v_i^\ast, \gamma_j}(v) - \sum_{\alpha^\vee}  q_\alpha(v) \Delta_{v_i^\ast, e_{\alpha^\vee} \gamma_j}(v) $$
for some polynomials $p, q_{\alpha^\vee}$ of degree $\mu_j-1$.  Comparing coefficients of powers of $v$,  it follows from the discussion in the previous paragraph that $\Delta_{v_i^\ast, f_j \gamma_j}(v) \in I_\mu^\lambda$ for
$$ s >  m_i + \langle \mu^*,\varpi_i^\vee -\text{wt}(f_j \gamma_j) \rangle  = m_i + \langle \mu^*, \varpi_i^\vee - \text{wt}(\gamma)\rangle $$
and so the same is true of $\Delta_{v_i^\ast, \gamma}^{(s)}$.  The proof that the remainder of the elements $\Delta_{\beta, \gamma}^{(s)}$ lie in the ideal as claimed now proceeds by taking Poisson brackets with elements $\Delta_{e_j v_j^\ast, v_j}^{(1)}$, exactly as in \cite[Proposition 2.15]{KWWY}.

Finally, by Lemma \ref{lemma: Poisson closed} below $\widetilde{I}_\mu^\lambda$ is closed under Poisson brackets with $\O(\Gr_\mu)$.  Since $\widetilde{I}_\mu^\lambda$ contains the Poisson generators of $J_\mu^\lambda$, it follows that $J_\mu^\lambda \subset \widetilde{I}_\mu^\lambda$.  Therefore $I_\mu^\lambda \subset\widetilde{I}_\mu^\lambda $, so we have equality.
\end{proof}

\begin{Lemma} \label{lemma: Poisson closed}
With notation as in the above proof, $\widetilde{I}_\mu^\lambda$ is closed under Poisson brackets with $\O(\Gr_\mu)$. 
\end{Lemma}
\begin{proof}
It suffices to verify that for $x$ one of the generators for $\O(\Gr_\mu)$ from Theorem \ref{subalgebra}, and for $y$ one of the generators of $\widetilde{I}_\mu^\lambda$, we have $\{x,y\} \in \widetilde{I}_\mu^\lambda$.  

We compute $\{ \Delta_{v_i^\ast,v_i}^{(r)}, \Delta_{\beta,\gamma}(v)\}$ as in Lemma \ref{lemma: bracket with f_i}.  Using equation (\ref{eq: poisson bracket of minors}),
\begin{align*} \{ \Delta_{v_i^\ast, v_i}^{(r)}, \Delta_{\beta,\gamma}(v)\} &= \res_u \left( u^{r-1}\frac{u^{-1}}{1-u^{-1}v} \bigg[ (\varpi_i^\vee, \text{wt}(\gamma) + \text{wt}(\beta) )\Delta_{v_i^\ast,v_i}(u) \Delta_{\beta,\gamma}(v) \right.\\
&\phantom{X}+\left. \sum_{{\alpha^\vee} \in \Phi_+^\vee} d_{\alpha^\vee} \Big( \Delta_{v_i^\ast, f_{\alpha^\vee} v_i}(u) \Delta_{\beta, e_{\alpha^\vee} \gamma}(v) - \Delta_{e_{\alpha^\vee} v_i^\ast, v_i}(u) \Delta_{f_{\alpha^\vee} \beta, \gamma}(v) \Big) \bigg] \right)
\end{align*}
Extracting the coefficient of $v^{-s}$ where $s> m_i + \langle \mu^*, \varpi_i^\vee - \text{wt}(\gamma)\rangle$, this expresses $\{\Delta_{v_i^\ast,v_i}^{(r)}, \Delta_{\beta,\gamma}^{(s)}\}$ as an element of $\widetilde{I}_\mu^\lambda$, since the series in $v$ on the right-hand side contribute generators from this ideal.  The case of $x = \Delta_{e_i v_i^\ast, v_i}^{(r)}$ is similar.

An analogous argument using Lemma \ref{lemma: bracket with f_i} shows that $\{f_i^{(s)}, \Delta_{v_i^\ast,\gamma}^{(r)}\} \in \widetilde{I}_\mu^\lambda$ for $s> \mu_i$ and $r> m_i + \langle \mu^*, \varpi_i^\vee - \text{wt}(\gamma)\rangle$.  To show that $\{f_i^{(s)}, \Delta_{\beta,\gamma}^{(r)}\} \in \widetilde{I}_\mu^\lambda$ for all $\beta$, we use induction on $\text{ht}(\beta+\varpi_i^\vee)$.  The base case $\beta = v_i^\ast$ is covered above.  For the inductive step, write $\beta = \sum_j e_j \beta_j$. Then as in the proof of \cite[Proposition 2.15]{KWWY},
$$ \Delta_{\beta,\gamma}^{(r)} = \sum_j \left( d_j^{-1} \{ \Delta_{e_j v_j^\ast, v_j}^{(1)}, \Delta_{\beta_j,\gamma}^{(r)}\} - \Delta_{\beta_j, e_j \gamma}^{(r)} \right) $$
It follows from \cite[Theorem 3.9]{KWWY} that the commutator of $\{f_i^{(s)},\cdot\}$ and $\{\Delta_{e_j v_j^\ast,v_j}^{(1)},\cdot\}$ is $\{\delta_{ij}h_i^{(s)},\cdot\}$, where $h_i^{(s)}$ is a certain polynomial in the leading minors $\Delta_{v_k^\ast,v_k}^{(t)}$.  Commuting these operators lets us apply the inductive hypothesis, proving the claim.
\end{proof}

Now that we have a precise description of $ I^\lambda_\mu $ we are in position to prove Theorem \ref{th:idealofY} and thus Theorem \ref{th:reducedness}.

\begin{proof}[Proof of Theorem \ref{th:idealofY}]
Let $ g $ be an $S$-point of $ \Gm $.  Assume that all the above generators of $ I_\mu^\lambda $ vanish on $ g $.  These conditions imply that $g$ is an $S$-point of $G_1[t^{-1}]$, which we identify with the big cell $ U \subset \Gr$. We would like to show that $ gt^{w_0 \mu} \in \overline{\Y^\lambda}(S) $.  Thus, we need to check condition (iii) appearing in the definition of $ \overline{\Y^\lambda} $.

First, we claim that it suffices to check this third condition when $ \nu^\vee $ is a fundamental weight $ \varpi_i^\vee $. Therefore we must check that for all $i$, the composed map
$$\phi_{\varpi_i^\vee} : \fF^0_G \times^G V(\varpi_i^\vee) \dashedrightarrow \fF_G \times^G V(\varpi_i^\vee) \rightarrow  \fF_G \times^G V(\varpi_i^\vee) \otimes \cO\left(\langle \lambda^*, \varpi_i^\vee \rangle \right)$$ is regular on all of $S \times X$, where $ \phi $ is the image of $ gt^{w_0 \mu} $ under the map $ G((t)) \rightarrow \Gr $.  Now $ gt^{w_0\mu} $ defines a $ \field((t)) \otimes \O(S) $-module map
$$
V(\varpi_i^\vee) \otimes \field((t)) \otimes \O(S) \rightarrow V(\varpi_i^\vee) \otimes \field((t)) \otimes \O(S)
$$
Condition (iii) is equivalent to the condition that the image of $ V(\varpi_i^\vee) \otimes \field[[t]] \otimes \O(S) $ lies in $ V(\varpi_i^\vee) \otimes t^{\langle -\lambda^*, \varpi^\vee_i \rangle} \C[[t]] \otimes \O(S) $.

If we pick a weight vectors $ \gamma \in V(\varpi^\vee_i) $ and $ \beta \in V(\varpi^\vee_i)^* $, we see that
$$ \Delta_{\beta, \gamma}(gt^{w_0\mu}) = t^{\langle w_0 \mu, \operatorname{wt}(\gamma) \rangle}\Delta_{\beta, \gamma}(g) \in t^{-\langle \lambda^*, \varpi_i^\vee \rangle} \field[[t]] $$
since the generators of $ I_\mu^\lambda $ vanish on $ g $.  Thus $ gt^{w_0 \mu}$ maps $ \gamma \otimes 1 \otimes 1 $ into the desired subspace, and since it is a module map, the result follows.

\end{proof}

\end{document}